\documentclass[11pt]{amsart}
\usepackage[mathscr]{eucal}
\usepackage{times}
\usepackage{amsmath,amssymb,latexsym,amscd}   
\usepackage{hyperref}
\hypersetup{colorlinks=true,linkcolor=blue,citecolor=red,linktocpage=true,pagebackref=true,hyperindex=true}
\usepackage{graphicx}
\usepackage[all,cmtip]{xy}
\usepackage{fancyhdr}
\usepackage{mathalfa}
\usepackage{mathrsfs}
\usepackage[sans]{dsfont}
\usepackage{upgreek}
\DeclareMathOperator{\tors}{tors}
\DeclareMathOperator{\Mod}{Mod}

\theoremstyle{plain}
\newtheorem{thm}{\protect\theoremname}[section]
  \theoremstyle{remark}
  
  \theoremstyle{remark}
  \newtheorem{ej}[thm]{\protect\examplename}
    \theoremstyle{remark}
  
  \theoremstyle{plain}
  \newtheorem{cor}[thm]{\protect\corollaryname}
  \theoremstyle{plain}
  \newtheorem{lem}[thm]{\protect\lemmaname}
  \theoremstyle{plain}
  \newtheorem{prop}[thm]{\protect\propositionname}
    \theoremstyle{definition}
  \newtheorem{dfn}[thm]{\protect\definitionname}
    \providecommand{\corollaryname}{Corollary}
  \providecommand{\examplename}{Example}
  \providecommand{\lemmaname}{Lemma}
  \providecommand{\propositionname}{Proposition}
  \providecommand{\remarkname}{Remark}
\providecommand{\theoremname}{Theorem}
\providecommand{\subexamplename}{Subexample}
  \providecommand{\definitionname}{Definition}
\begin{document}
%\begin{frontmatter}

\title{Dimension and Decomposition in modular upper-continuous lattices}
\author{Jos\'e R\'\i os Montes}% \corref{cor1}}
\address{Instituto de Matem\'aticas, Ciudad Universitaria, UNAM, M\'exico, D. F., 04510, M\'exico.}
\email{jrios@matem.unam.mx,  zaldivar@matem.unam.mx}
%\ead{j.rios@matem.unam.mx}
\author{Angel Zald\'\i var Corichi}%\corref{cor2}}
%\email{zaldivar@matem.unam.mx}
%\ead{zaldivar@matem.unam.mx}
%$\cortext[cor1]{Corresponding author}
\maketitle
\begin{abstract}
We translate notions and results of decomposition and dimension theories for module categories,  into the lattice environment. In particular we translate dimension theory in module categories to complete modular upper-continuous lattices. 
\end{abstract}
%\begin{keyword}
%Lattices\sep inflators \sep preradicals \sep radicals  
%\MSC[2010]  06B99\sep 06C\sep 16P20\sep 16S90\sep 18E40
%\end{keyword}
%\end{frontmatter}
%\linenumbers

%\section{Preamble}

\section{Introduction}\label{sec:sec1}

The notions of dimension and decomposition for module categories have been extensively studied for many authors from different perspectives. Starting with the commutative case,  the notions of primary decompositions and Krull dimension,  have been extended to the non-commutative setting. Moreover, these constructions have also been extended to an arbitrary abelian category.  Good accounts for these developments are \cite{9,12,14,22}. The book \cite{7} organizes all the distinct decompositions and dimensions in module categories and gives a general point of view for the treatment of these theories via \emph{radical functions}, \emph{quasi-decomposition functions} and \emph{quasi-dimension functions}. Most of these treatments use lattice concepts for the particular case of the lattice of all torsion theories  \cite{8}. Later on, in \cite{25} the author describes an analogue treatment of decompositions for modules of \cite{7} for complete, modular, meet-continuous (upper-continuous) lattices  via \emph{allocations}.  In the same setting as \cite{7,25} we develop the general dimension theory that is not developed in \cite{7}. 
The organization of the paper is as follows:
Section \ref{sec:sec2} gives the general background necessary for most of the paper. In Section \ref{sec:sec3} we develop the general setting of allocations and we introduce the concept of \emph{aspect} for complete modular upper-continuous lattice . We also  investigate some properties and relations with allocations. Section  \ref{sec:sec4} is an account based in some results of \cite{25} with some generalizations. Section  \ref{sec:sec5} describes the notion of dimension for complete modular upper-continuous lattices. We  prove that this notion is exactly the analogue for module categories via filtrations of torsion theories. In Section \ref{sec:sec6} we prove that the concept of quasi-dimension function in a module category is intrinsically linked with the concept of aspect for the lattice of submodules of a given module $M$.

\section{Preliminaries and background material}\label{sec:sec2}

An \emph{idiom} $(A,\leq,\bigvee,\wedge,\bar{1},\underline{0})$ is a complete, upper-continuous, modular lattice, that is, $A$ is a complete lattice that satisfies the following distributive laws: 
$$a\wedge (\bigvee X)=\bigvee\left\{a\wedge x\;|\; x\in X\right\}\leqno({\rm IDL})$$ 
holds for all $a\in A$  and $X\subseteq A$ directed, and 
$$(a\vee c)\wedge b=a\vee(c\wedge b)\leqno({\rm ML})$$ 
for all $a,b,c\in A$. 
These lattice were introduced in \cite{23}, and a more recent account is in \cite{26}. We also need the following class of idioms: A \emph{frame} $(A, \leq, \bigvee, \wedge, \bar{1}, \underline{0})$ is a complete lattice that satisfies
$$a\wedge (\bigvee X)=\bigvee\left\{a\wedge x\;|\; x\in X\right\}\leqno({\rm FDL})$$
for all $a\in A$ and $X\subseteq A$ any subset.
 Two fundamental examples are the following: 
Given a ring $R$ and any left $R$-module $M$, the lattice $\text{Sub}_{R}(M)$ of all submodules of $M$ is modular and upper-continuous, hence it is an idiom. 
Frames are the algebraic version of a topological space. Indeed, if $S$ is a topological space then its topology, $\mathcal{O}(S)$ is a frame.  The correspondence $S\mapsto \mathcal{O}(S)$ has been extensively studied, for example see \cite{11} and \cite{13}. It is
 important to mention that frames are characterized by an implication. Recall that in any lattice $A$, an \emph{implication} in $A$ is an operation $(\_\succ \_)$ given by $x\leq (a\succ b)\Leftrightarrow x\wedge b\leq a$, for all $a,b\in A$. For a proof of the following fact, see \cite{24}.

\begin{prop}\label{03}
A complete lattice $A$ is a frame if and only if $A$ has an implication.
\end{prop}

%To describe the structure of an idiom there are certain kind of functions that capture much of the information about the %complexity of the idiom.

We will use the following concepts. An \emph{inflator} on an idiom $A$ is a function $d:A\rightarrow A$ such that $x\leq d(x)$ and  $x\leq y \Rightarrow d(x)\leq d(y)$. 
A \emph{pre-nucleus} $d$ on $A$ is an inflator such that $d(x\wedge y)=d(x)\wedge d(y)$. 
A \emph{stable} inflator on $A$ is an inflator such that $d(x)\wedge y\leq d(x\wedge y)$ for all $x,y\in A$.   
Let  $I(A)$ denote the set of all inflators on $A$, $P(A)$ the set of all prenuclei, and $S(A)$ the set of all stable inflators. Clearly, $P(A)\subseteq S(A)\subseteq I(A)$. 
A \emph{closure operator} is an idempotent inflator $c$ on $A$, that is,  is an inflator such that $c^{2}=c$. Let $C(A)$ the set of all closure operators in $A$. A \emph{nucleus} on $A$ is a idempotent pre-nucleus. Let  $N(A)$ be the set of all nuclei on $A$. All these sets are partially ordered by $d\leq f\Leftrightarrow d(a)\leq f(a)$ for all $a\in A$. Note that the identity function $id_{A}$ and the constant function $\bar{d}(a)=\bar{1}$ for all $a\in A$ (where $\bar{1}$ is the top of $A$) are inflators. These two inflators are the bottom and the top in all these partially ordered sets.  Moreover, we can describe the infimum $\bigwedge\mathcal{I}$ of any subset $\mathcal{I}\subseteq L$, for $L\in\left\{I(A),P(A),S(A),C(A),N(A)\right\}$, as the function on $A$ given by $(\bigwedge\mathcal{I})(a)=\bigwedge\left\{f(a) |f\in\mathcal{I}\right\}$ for each $a\in A$. It is immediate that this function lies in $L$, and is the infimum of the family $\mathcal{I}$. Therefore, each of these sets is a complete lattice.

Inflators tell us something about the complexity of the idiom. Indeed, given an inflator $d\in I(A)$, let $d^{0}:=id_{A}$,  $d^{\alpha+1}:=d\circ d^{\alpha}$ for a non-limit ordinal $\alpha$, and let  $d^{\lambda}:=\bigvee\left\{d^{\alpha}| \alpha<\lambda\right\}$ for a limit ordinal $\lambda$. These are inflators, ordered in a chain 
$$d\leq d^{2}\leq d^{3}\leq\ldots \leq d^{\alpha}\leq\ldots.$$ 
By a cardinality argument, there exists an ordinal $\gamma$ such that $d^{\alpha}=d^{\gamma}$, for $\alpha\geq \gamma$. In fact, we can choose $\gamma$ the least of these ordinals, say $\infty$. Thus, $d^{\infty}$ is an inflator such that $d\leq d^{\infty}$, but more important this inflator satisfies $d^{\infty}d^{\infty}=d^{\infty}$, that is, $d^{\infty}$ is a  closure operator on $A$. We say that an idiom $A$ has $d$-\emph{length} if $d^{\infty}(\underline{0})=\bar{1}$, that is, the associated idempotent $d^{\infty}$ is just the top of $I(A)$. In order to have a workable notion of dimension one has to work with the complete lattice on all nuclei:

%one can say that the dimension of $A$ with respect $d$ is the ordinal $\infty$ if $d^{\infty}=\bar{d}$ but this terminology %leads to confusion, to talk about dimension of idioms we need the complete lattice $N(A)$ of all nuclei, one of this reasons %(and maybe the most important) is that:

\begin{thm}
\label{0}
For any idiom $A$, the complete lattice of all nuclei $N(A)$ in $A$ is a frame.
\end{thm}

A proof of this fact can be found in \cite{23,24,26}.
Another important fact about nuclei is that any element $j\in N(A)$ gives a quotient of $A$, the set $A_{j}$ of elements fixed by $j$.  Even more,  $A_{j}$ is an idiom, and thus many properties of $A$ are reflected in $A_{j}$. 
Since $N(A)$ is an idiom, it has his own inflators,  and we may consider any stable inflator $S$ over $N(A)$. Following Simmons we said that a nucleus $j$ over $A$ has $S$-\emph{dimension} if $S^{\infty}(j)=\bar{d}$.  In particular, for the nucleus $id_A$ of $A$, since $A_{id_A}=A$, if $id_A$ has $S$-dimension $\theta$, then we  say that the $S$-dimension of $A$ is $\theta$.  This is actually the central idea of dimension: Given a property in the idiom $A$, this property gives a stable inflator $S$, and we want to measure how far $A$ or some quotient $A_j$, with respect to the nucleus $j$, has the property;   that measure is the ordinal $\theta$. To organize all these, there is a frame, the \emph{base} frame of the idiom $A$. 

Next, following Simmons \cite{23}, we  review  the construction of the base frame and other special frames used for the ranking and dimension of idioms. 
If $A$ is an idiom and $a,b\in A$ satisfy $a\leq b$, the \emph{interval} $[a,b]$ is the set $[a,b]=\left\{x\in A\;|\; a\leq x \leq b\right\}$. Denote by $\EuScript{I}(A)$ the set of all intervals of $A$. Given two intervals $I, J$, we say that $I$ is a \emph{subinterval} of $J$, denoted by $I\rightarrow J$, if $I\subseteq J$, that is, if $I=[a,b]$ and $J=[a',b']$ with $a'\leq a\leq b\leq b'$ in $A$. We say that $J$ and $I$ are \emph{similar}, denoted by  $J\sim I$, if there are $l,r\in A$ with associated intervals $$L=[l,l\vee r]\;\;\;\; [l\wedge r,r]=R$$ where $J=L$ and $I=R$ or $J=R$ and $I=L$. Clearly, this a reflexive and symmetric relation. Moreover, if $A$ is modular, this relation is just the canonical lattice isomorphism between $L$ and $R$. 

We say that a set of intervals $\mathcal{A}\subseteq {\EuScript I}(A)$ is \emph{abstract} if is not empty and it is closed under $\sim$, that is, 
$$J\sim I\in\mathcal{A}\Rightarrow J\in\mathcal{A}.$$ 
An abstract set $\mathcal{B}$ is a \emph{basic} set of intervals if it is closed by subintervals, that is, 
$$J\rightarrow I\in\mathcal{B}\Rightarrow J\in\mathcal{B}$$ 
for all intervals $I,J$. A set of intervals $\mathcal{C}$ is a \emph{congruence} set if it is basic and closed under abutting intervals, that is, 
$$[a,b][b,c]\in \mathcal{C}\Rightarrow [a,c]\in\mathcal{C}$$
for elements $a,b,c\in A$. A basic set of intervals $\mathcal{B}$ is a \emph{pre-division} set if $$\forall\; x\in X\left[[a,x]\in\mathcal{B}\Rightarrow [a,\bigvee X]\in\mathcal{B}\right]$$ for each $a\in A$ and $X\subseteq [a,\bar{1}]$. A set of intervals $\mathcal{D}$ is a \emph{division} set if it is a congruence set and a pre-division set.
Put $\EuScript{D}(A)\subseteq\EuScript{C}(A)\subseteq\EuScript{B}(A)\subseteq\EuScript{A}(A)$ the set of all division, congruence, basic and abstract set of intervals in $A$. This gadgets can be understood like certain classes of modules in a module category $R$-$\Mod$, that is, classes closed under isomorphism, subobjects, extensions and coproducts. From this point of view $\EuScript{C}(A)$ and $\EuScript{D}(A)$ are the idioms analogues of the Serre classes and the torsion (localizations) classes in module categories.

Note that $\EuScript{B}(A)$ is closed under arbitrary intersections and unions, hence it is a frame. The top of this frame is $\EuScript{I}(A)$ and the bottom is the set of all trivial intervals of $A$, denoted by $\EuScript{O}(A)$ or simply by $\EuScript{O}$. The frame $\EuScript{B}(A)$ is the \emph{base} frame of the idiom $A$. 

The family $\EuScript{C}(A)$ is closed under arbitrary intersections, but suprema are not unions; to describe the suprema we take any basic set $\mathcal{B}$ and the least congruence set that contains it, this usual construction leads to a inflator over the base frame $\EuScript{B}(A)$ as follows: 
For each  $\mathcal{B}\in\EuScript{B}(A)$, let $\EuScript{C}ng(\mathcal{B})$ be the set of all intervals $[a,b]$ which can be partitioned by $\mathcal{B}$, that is, there is a finite chain $a=x_{0}\leq\ldots\leq x_{i}\leq\ldots\leq x_{m+1}=b$ such that $[x_{i},x_{i+1}]\in\mathcal{B}$ for each $0\leq i\leq m$. 
Note that, by definition, $\mathcal{B}\subseteq \EuScript{C}ng(\mathcal{B})$.  As in Lemma 5.3 of \cite{23} for each basic set $\mathcal{B}$, $\EuScript{C}ng(\mathcal{B})$ is the least congruence set that includes $\mathcal{B}$. Moreover, $\EuScript{C}ng(\_)$ is a nucleus over the frame $\EuScript{B}(A)$ with fixed set $\EuScript{C}(A)$, that is, $\EuScript{C}(A)$ is a frame. 

 An interval $[a,b]$ is \emph{simple} if there is no $a< x< b$ that is $[a,b]=\left\{a,b\right\}$. Denote by $\EuScript{S}mp$ be the set of all simple intervals.  An interval$[a,b]$ of $A$ is \emph{complemented} if it is a complemented lattice, that is, for each $a\leq x\leq b$ there exist $a\leq y\leq b$ such that $a=x\wedge y$ and  $b=x\vee y$.  Let $\EuScript{C}mp$ be the set of all complemented intervals. In fact, for every $\mathcal{B}$ we can define $\EuScript{S}mp(\mathcal{B})$ and $\EuScript{C}mp(\mathcal{B})$: the former is the set of intervals that are $\mathcal{B}$-simple, that is, the set of all $[a,b]$ such that for each $a\leq x\leq b$, $[a,x]\in\mathcal{B}$ or $[x,b]\in\mathcal{B}$, and the latter is the set of all intervals that are $\mathcal{B}$-complemented, that is, $[a,b]$ such that for every $a\leq x\leq b$ exists $a\leq y\leq b$ such that $[a,x\wedge y]\in\mathcal{B}$ and $[x\vee y,b]\in\mathcal{B}$.  With this, we have that $\EuScript{S}mp=\EuScript{S}mp(\EuScript{O})$ and $\EuScript{C}mp=\EuScript{C}mp(\EuScript{O})$.

There are others special sets of intervals: Given any $\mathcal{B}\in\EuScript{B}(A)$ denote by $\EuScript{C}rt(\mathcal{B})$ the set of intervals such that for all $a\leq x\leq b$ we have $a=x$ or $[x,b]\in\mathcal{B}$; this is the set of all $\mathcal{B}$-\emph{critical} intervals. Denote now by $\EuScript{F}ll(\mathcal{B})$ the set of all intervals $[a,b]$ such that, for all $a\leq x\leq b$ there exists $a\leq y\leq b$ with $a=x\wedge y$ and $[x\vee y,b]\in\mathcal{B}$; this is the set of all $\mathcal{B}$-\emph{full} intervals.  Note that $\EuScript{S}mp(\EuScript{O})=\EuScript{C}rt(\EuScript{O})$ and $\EuScript{C}mp(\EuScript{O})=\EuScript{F}ll(\EuScript{O})$. In \cite{23} Simmons proves that for any $\mathcal{B}\in\EuScript{B}(A)$, $\EuScript{C}rt(\mathcal{B})\leq \EuScript{S}mp(\mathcal{B})$, $\EuScript{F}ll(\mathcal{B})\leq \EuScript{C}mp(\mathcal{B})$, $\EuScript{S}mp(\mathcal{B})\leq \EuScript{C}mp(\mathcal{B})$ and $\EuScript{C}rt(\mathcal{B})\leq \EuScript{F}ll(\mathcal{B})$.  Moreover, he shows that for  any $\mathcal{B}\in\EuScript{B}(A)$ the sets 
$\EuScript{S}mp(\mathcal{B}), \EuScript{C}mp(\mathcal{B})$ and $\EuScript{C}rt(\mathcal{B}), \EuScript{F}ll(\mathcal{B})$
are basic. He also proves that $\EuScript{S}mp(\_), \EuScript{C}mp(\_)\in S(\EuScript{B}(A))$ and $\EuScript{C}rt(\_), \EuScript{F}ll(\_)\in P(\EuScript{B}(A))$  

For the set $\EuScript{D}(A)$ and for any $\mathcal{B}\in\EuScript{B}(A)$ we can describe the least division set that contains it. Since $\EuScript{D}(A)$ is closed under arbitrary intersections, denote by $\EuScript{D}vs(\mathcal{B})$ that division set that contains it. In \cite{23} it is proved that $\EuScript{D}vs(\_)$ is a nucleus over $\EuScript{B}(A)$ and the quotient of this nucleus is $\EuScript{D}(A)$. In fact, there is a relation with this frame and the frame $N(A)$: To describe this relation, take any basic set $\mathcal{B}$ and $a\in A$; define $|\mathcal{B}|(a)=\bigvee X$, where $x\in X\Leftrightarrow [a,x]\in\mathcal{B}$. This produces the associated inflator of $\mathcal{B}$. Moreover, if the basic set $\mathcal{B}$ is a congruence set, then $|\mathcal{B}|$ is a pre-nucleus in $A$, and if it is a division set, then $|\mathcal{B}|$ is a nucleus. In this way we have  for every division set  a nucleus. Now, given a nucleus $j$ we can construct a division set $[a,b]\in\mathcal{D}_{j}\Leftrightarrow j(a)=j(b)$. These correspondences are bijections and moreover they define an isomorphism between $\EuScript{D}(A)$ and $N(A)$, with this we have:

\begin{thm}
\label{00}
If  $A$ is an idiom, then there is an isomorphism of frames $$N(A)\longleftrightarrow \EuScript{D}(A)$$ 

$$j\longleftrightarrow \mathcal{D}$$ given by 
$$j\longmapsto \mathcal{D}_{j}\;\;\; [a,b]\in\mathcal{D}_{j}\Longleftrightarrow b\leq j(a)$$ 
$$ \mathcal{D}\longmapsto\;\;\;j=|\mathcal{D}|$$
\end{thm}

The $\EuScript{D}vs$-construction can be described it in a useful way: 

\begin{thm}
\label{000}
For every $\mathcal{B}\in\EuScript{B}(A)$ we have $$[a,b]\in\EuScript{D}vs(\mathcal{B})\Longleftrightarrow (\forall a\leq x< b)(\exists x<y\leq b)[[x,y]\in\mathcal{B}],$$ 
for each interval $[a,b]$.
\end{thm}

Details  are in \cite{23}, and a more recent account is given in \cite{27} and \cite{28}. 

\begin{ej} An interval $[a,b]$ is \emph{simple} if there is no $a< x< b$ that is $[a,b]=\left\{a,b\right\}$. Denote by $\EuScript{S}mp$ be the set of all simple intervals.  An interval$[a,b]$ of $A$ is \emph{complemented} if it is a complemented lattice, that is, for each $a\leq x\leq b$ there exist $a\leq y\leq b$ such that $a=x\wedge y$ and  $b=x\vee y$. Let $\EuScript{C}mp$ be the set of all complemented intervals. In fact, for every $\mathcal{B}$ we can define $\EuScript{S}mp(\mathcal{B})$ and $\EuScript{C}mp(\mathcal{B})$: the former is the set of intervals that are $\mathcal{B}$-simple, that is, the set of all $[a,b]$ such that for each $a\leq x\leq b$, $[a,x]\in\mathcal{B}$ or $[x,b]\in\mathcal{B}$, and the latter is the set of all intervals that are $\mathcal{B}$-complemented, that is, $[a,b]$ such that for every $a\leq x\leq b$ exists $a\leq y\leq b$ such that $[a,x\wedge y]\in\mathcal{B}$ and $[x\vee y,b]\in\mathcal{B}$.  With this, we have that $\EuScript{S}mp=\EuScript{S}mp(\EuScript{O})$ and $\EuScript{C}mp=\EuScript{C}mp(\EuScript{O})$.

There are others special sets of intervals: Given any $\mathcal{B}\in\EuScript{B}(A)$ denote by $\EuScript{C}rt(\mathcal{B})$ the set of intervals such that for all $a\leq x\leq b$ we have $a=x$ or $[x,b]\in\mathcal{B}$; this is the set of all $\mathcal{B}$-\emph{critical} intervals. Denote now by $\EuScript{F}ll(\mathcal{B})$ the set of all intervals $[a,b]$ such that, for all $a\leq x\leq b$ there exists $a\leq y\leq b$ with $a=x\wedge y$ and $[x\vee y,b]\in\mathcal{B}$; this is the set of all $\mathcal{B}$-\emph{full} intervals. Note that $\EuScript{S}mp(\EuScript{O})=\EuScript{C}rt(\EuScript{O})$ and $\EuScript{C}mp(\EuScript{O})=\EuScript{F}ll(\EuScript{O})$. In \cite{23} is proved that for any $\mathcal{B}\in\EuScript{B}(A)$, $\EuScript{C}rt(\mathcal{B})\leq \EuScript{S}mp(\mathcal{B})$, $\EuScript{F}ll(\mathcal{B})\leq \EuScript{C}mp(\mathcal{B})$, $\EuScript{S}mp(\mathcal{B})\leq \EuScript{C}mp(\mathcal{B})$ and $\EuScript{C}rt(\mathcal{B})\leq \EuScript{F}ll(\mathcal{B})$. Moreover, one can shows that for any $\mathcal{B}\in\EuScript{B}(A)$ the sets 
$\EuScript{S}mp(\mathcal{B}), \EuScript{C}mp(\mathcal{B})$ and $\EuScript{C}rt(\mathcal{B}), \EuScript{F}ll(\mathcal{B})$
are basic. Also it can be seen that $\EuScript{S}mp(\_), \EuScript{C}mp(\_)\in S(\EuScript{B}(A))$ and $\EuScript{C}rt(\_), \EuScript{F}ll(\_)\in P(\EuScript{B}(A))$  
\end{ej}
%%%%%%%%%%%%%%%%%%%%%%%%%%%%%%%%%%%%%%%%%%%%%%%%

\section{$\Lambda$-allocations and $\Lambda$-aspects}\label{sec:sec3}

In \cite{25} the author introduces the concept of $\Lambda$-allocation for an idiom $A$ to study the decomposition of intervals on $A$.  This concept can be understood as the idiomatic version of the decomposition theory  in \cite{7}.

\begin{dfn}\label{d1}
If $\Lambda$ is a complete lattice, for an idiom $A$ a $\Lambda$-\emph{allocation} is a function $\varphi: \EuScript{I}(A)\longrightarrow\Lambda$ that satisfies the following:
\begin{enumerate}
\item $\varphi(l\wedge r,r)=\varphi(l,l\vee r)$ with $r,l\in A$.

\item $\varphi(a,b)\leq\varphi(a,c)$  for $a\leq c\leq b$.

\item $\varphi(a,c)\wedge \varphi(c,b)\leq \varphi(a,b)$  for $a\leq c\leq b$.

\item $\varphi(a,\bigvee X)=\bigwedge\left\{\varphi(a,x)\;|\; x\in X\right\}$ for $a\in A$ and $X\subseteq [a,\bar{1}]$ directed.
\end{enumerate}
\end{dfn}

 In item (4) of Definition \ref{d1} the subset $X$ can be independent over $a$, this is pointed in $6.2$ of \cite{25}. 
In \cite{25} Simmons shows that for any idiom $A$ there is always a $N(A)$-allocation given by $\chi:\EuScript{I}(A)\rightarrow N(A)$, where $\chi(a,b)$ is the unique largest nucleus that  satisfies $\chi(a,b)(a)\wedge b=a$.

Denote by $\rm{Sit}(A,\Lambda)=\left\{\varphi\;|\; \varphi \textrm{ is a } \Lambda\textrm{-allocation} \right\}$. It is almost immediate that $\rm{Sit}(A,\Lambda)$ is a poset.  Moreover, it is a complete lattice.
Now, take any $f:A\longrightarrow A'$ idiom morphism and $\varphi\in \rm{Sit}(A',\Lambda)$.  Then, consider the induced poset morphism $\EuScript{I}(f):\EuScript{I}(A)\rightarrow\EuScript{I}(A')$ and the composition $\varphi\circ\EuScript{I}(f):\EuScript{I}(A)\rightarrow \Lambda$. From the definition and the fact that $f$ is monotone we have that $\varphi\circ\EuScript{I}(f)\in\rm{Sit}(A,\Lambda)$. Setting $f^{*}:\rm{Sit}(A',\Lambda)\rightarrow \rm{Sit}(A,\Lambda)$  the following is straightforward.

\begin{prop}
\label{d2}
Let  be $\Lambda$ a complete lattice. Then, $\rm{Sit}(\_, \Lambda):\EuScript{I}\EuScript{D}\longrightarrow \mathfrak{C}\mathfrak{L}$ is a contravariant functor from the category of idioms $\EuScript{I}\EuScript{D}$ to the category of complete lattices$\mathfrak{C}\mathfrak{L}$. 
\end{prop}
\qed

Now, fixing the first component consider any morphism between two complete lattices, $\varrho:\Lambda\rightarrow \Gamma$.  Thus, for any $\Lambda$-allocation $\varphi$ consider the composition $\varrho\circ\varphi:\EuScript{I}(A)\rightarrow \Gamma$.  It  is easy to see that this function is a $\Gamma$-allocation, and thus:

\begin{prop}
\label{d3}
If $A$ is any idiom, then $\rm{Sit}(A,\_):\mathfrak{C}\mathfrak{L}\longrightarrow\mathfrak{C}\mathfrak{L}$ is a covariant endofunctor in the category of complete lattices and preserves monomorphisms.
\end{prop}
\qed
%tal vez dar un poco mas de detalles de las situaciones anteriores.

\begin{prop}
\label{d4}
Let be $\Lambda$ a complete lattice and $A$ an idiom. Then, there is a function $\EuScript{Q}:\rm{Sit}(A,\Lambda)\times \Lambda\rightarrow\EuScript{C}(A)$ given by $[a,b]\in \EuScript{Q}(\varphi,\alpha)\Leftrightarrow \alpha\leq\varphi(x,b)$ for all $x\in[a,b]$. Moreover 

\begin{itemize}
\item[\rm{(1)}] For every $\varphi\in\rm{Sit}(A,\Lambda)$, the function $\EuScript{Q}(\varphi,\_):\Lambda^{\rm op}\rightarrow \EuScript{C}(A)$ is a $\bigwedge$-morphism in $\mathfrak{C}\mathfrak{L}$.
\item[\rm{(2)}] For every $\alpha\in\Lambda$, the function $\EuScript{Q}(\_,\alpha):\rm{Sit}(A,\Lambda)\rightarrow \EuScript{C}(A)$ is a $\bigwedge$-morphism in $\mathfrak{C}\mathfrak{L}$.
\end{itemize}
\end{prop}

\begin{proof}
First we show that $ \EuScript{Q}(\varphi,\alpha)\in\EuScript{B}(A)$ for every $\varphi\in\rm{Sit}(A,\Lambda)$ and every $\alpha\in\Lambda$.
From $(1)$ of Definition \ref{d1} it follows that $\EuScript{Q}(\varphi,\alpha)$ is abstract. Now, let $[a,d]\in\EuScript{Q}(\varphi,\alpha)$ and take $a\leq b\leq c\leq d$.  For any $b\leq x\leq c$, we have that $\varphi(x,d)\geq\alpha$.  But from $(2)$ of Definition \ref{d1} we have that $\varphi(x,d)\leq \varphi(x,c)$, that is, $[b,c]\in\EuScript{Q}(\varphi,\alpha)$.
Now consider $[a,b],[b,c]\in\EuScript{Q}(\varphi,\alpha)$, and let $x\in[a,c]$.  Thus, $a\leq b\wedge x\leq b$ and $b\leq x\vee b\leq c$.  Hence, by the hypothesis, $\varphi(x\wedge b,b)\geq \alpha$ and $\varphi(b\vee x,c)\geq \alpha$. On the other hand, we have $\varphi(x,c)\geq\varphi(x,b\vee x)\wedge \varphi(b\vee x,c)$, and  from modularity we deduce that $[x\wedge b,b]\cong [x,x\vee b]$. The last inequality, by $(1)$ of Definition \ref{d1}, is $\varphi(x,b\vee x)\wedge \varphi(b\vee x,c)=\varphi(x\wedge b,b)\wedge \varphi(b\vee x,c)$. But this infimum is above $\alpha$, and thus $\varphi(x,c)\geq\alpha$, that is, $\EuScript{Q}(\varphi,\alpha)\in\EuScript{C}(A)$.

Now, to prove part (1), note that if $\alpha\leq\alpha'$ in $\Lambda$, then $\EuScript{Q}(\varphi,\alpha')\leq\EuScript{Q}(\varphi,\alpha)$ by definition.  Thus, if $X\subseteq \Lambda$ we have that $\EuScript{Q}(\varphi,\bigvee X)\leq\bigcap\left\{\EuScript{Q}(\varphi,\alpha)|\; \alpha\in X \right\}$. But, for any $[a,b]\in\bigcap\left\{\EuScript{Q}(\varphi,\alpha)|\; \alpha\in X \right\}$ we have that $\varphi(x,b)\geq \alpha $ for every $\alpha\in X$. Therefore, $\varphi(x,b)\geq \bigvee X$, that is, $[a,b]\in\EuScript{Q}(\varphi,\bigvee X)$, and hence $\bigcap\left\{\EuScript{Q}(\varphi,\alpha)|\; \alpha\in X \right\}=\EuScript{Q}(\varphi,\bigvee X)$.

For part (2), observe that if $\varphi\leq\varphi'$, then $\EuScript{Q}(\varphi,\alpha)\leq\EuScript{Q}(\varphi',\alpha)$.   The result is immediate now.
\end{proof}

We know that $\EuScript{D}vs(\_):\EuScript{B}(A)\rightarrow \EuScript{B}(A)$ is a nucleus, so $\EuScript{D}vs(\_):\EuScript{C}(A)\rightarrow \EuScript{D}(A)$ is a frame morphism, in particular a $\wedge$-morphism.

\begin{cor}
\label{d5}
Let be $\Lambda$ a complete lattice and $A$ an idiom. Then,
\begin{itemize}
\item[\rm{(1)}] For every $\varphi\in\rm{Sit}(A,\Lambda)$, the function $\EuScript{D}vs(\_)\circ\EuScript{Q}(\varphi,\_):\Lambda^{op}\rightarrow \EuScript{D}(A)$ is a $\wedge$-morphism in $\mathfrak{C}\mathfrak{L}$.
\item[\rm{(2)}] For every $\alpha\in\Lambda$, the function $\EuScript{D}vs(\_)\circ\EuScript{Q}(\_,\alpha):\rm{Sit}(A,\Lambda)\rightarrow \EuScript{D}(A)$ is a $\wedge$-morphism in $\mathfrak{C}\mathfrak{L}$.
\end{itemize}
\end{cor}

\begin{proof}
Direct from Proposition \ref{d4} and the previous observation.
\end{proof}

For every idiom $A$ and any  complete lattice $\Lambda$, we have a function $\EuScript{S}:\Lambda\rightarrow \rm{Sit}(A,\Lambda)$ defined by $\EuScript{S}(\alpha)(a,b)=\alpha$. The following is straightforward.

\begin{prop}
\label{d6}
Let $\Lambda$ be a complete lattice and $A$ an idiom. Then, the function $\EuScript{S}:\Lambda\rightarrow \rm{Sit}(A,\Lambda)$ defined by $\EuScript{S}(\alpha)(a,b)=\alpha$, is an embedding in the category of complete lattices. 
\end{prop}

\begin{dfn}\label{d7}
Let be $\Lambda$ a complete lattice.  For an idiom $A$, a $\Lambda$-\emph{aspect} is a function $\varphi: \EuScript{I}(A)\longrightarrow\Lambda$ that satisfies the following:
\begin{enumerate}

\item $\varphi(l\wedge r,r)=\varphi(l,l\vee r)$, for $r,l\in A$.

\item $\varphi(a,c)\vee \varphi(c,b)= \varphi(a,b)$,  for $a\leq c\leq b$.

\item $\varphi(a,\bigvee X)=\bigvee\left\{\varphi(a,x)\;|\; x\in X\right\}$, for $a\in A$ and $X\subseteq [a,\bar{1}]$ directed.
\end{enumerate}
\end{dfn}

Denote by $\rm{App}(A,\Lambda)=\left\{\varphi\;|\; \varphi \textrm{ is a } \Lambda\textrm{-aspect} \right\}$. It is immediate that $\rm{App}(A,\Lambda)$ is a poset and a complete lattice.
Take any  idiom morphism  $f:A\longrightarrow A'$ and $\varphi\in \rm{App}(A',\Lambda)$.  Consider the induced poset morphism $\EuScript{I}(f):\EuScript{I}(A)\rightarrow\EuScript{I}(A')$. Then, $\varphi\circ\EuScript{I}(f):\EuScript{I}(A)\rightarrow \Lambda$, and from the definition and the fact that $f$ is monotone we have $\varphi\EuScript{I}(f)\in\rm{App}(A,\Lambda)$. For $f^{*}:\rm{App}(A',\Lambda)\rightarrow \rm{App}(A,\Lambda)$  the following is immediate.

\begin{prop}
\label{d8}
Let  be $\Lambda$ a complete lattice. Then, $\rm{App}(\_, \Lambda):\EuScript{I}\EuScript{D}\longrightarrow \mathfrak{C}\mathfrak{L}$ is a contravariant functor from the category of idioms $\EuScript{I}\EuScript{D}$ to the category of complete lattices $\mathfrak{C}\mathfrak{L}$. \qed
\end{prop}

Consider now any morphism between two complete lattices, $\varrho:\Lambda\rightarrow \Gamma$.  Thus, for any $\Lambda$-aspect $\varphi$, we have $\varrho\circ\varphi:\EuScript{I}(A)\rightarrow \Gamma$, and it is easy to see that this function is a $\Gamma$-aspect. We have shown:
\begin{prop}
\label{d9}
If $A$ is any idiom then, $\rm{App}(A,\_):\mathfrak{C}\mathfrak{L}\longrightarrow\mathfrak{C}\mathfrak{L}$ is a covariant endofunctor in the category of complete lattices and  preserves monomorphisms.
\end{prop}
\qed

\begin{prop}
\label{d10}
Let be $\Lambda$ a complete lattice and $A$ an idiom. Then, there is a function $\EuScript{M}:\rm{App}(A,\Lambda)\times \Lambda\rightarrow\EuScript{C}(A)$ given by $[a,b]\in \EuScript{M}(\varphi,\alpha)\Leftrightarrow \varphi(a,b)\leq\alpha$ that satisfies

\begin{itemize}
\item[\rm{(1)}] For every $\varphi\in\rm{App}(A,\Lambda)$, the function $\EuScript{M}_{\varphi}:\Lambda\rightarrow \EuScript{C}(A)$ is a $\bigwedge$-morphism in $\mathfrak{C}\mathfrak{L}$.
\item[\rm{(2)}] For every $\alpha\in\Lambda$, the function $\EuScript{M}_{\alpha}:\rm{App}(A,\Lambda)\rightarrow \EuScript{C}(A)$ is a $\bigwedge$-morphism in $\mathfrak{C}\mathfrak{L}$.
\end{itemize}
\end{prop}

\begin{proof}
Take $(\varphi,\alpha)\in\rm{App}(A,\Lambda)$. By (1) of Definition \ref{d7} we have that $\EuScript{M}(\varphi,\alpha)$ is abstract. Now, consider $[a,d]\in\EuScript{M}(\varphi,\alpha)$ and $a\leq b\leq c\leq d$. Then by (2) of Definition \ref{d7} we have that $\varphi(a,c)\leq\varphi(a,d)\leq\alpha$, and again by (3) of Definition \ref{d7}, $\varphi(a,c)=\varphi(a,b)\vee\varphi(b,c)\leq\alpha$. Therefore, $\varphi(b,c)\leq\alpha$, and hence $\EuScript{M}(\varphi,\alpha)$ is basic.
Now, take $[a,b],[b,c]\in\EuScript{M}(\varphi,\alpha)$. Then, $\alpha\geq\varphi(a,b)\vee\varphi(b,c)=\varphi(a,c)$, that is, $[a,c]\in\EuScript{M}(\varphi,\alpha)$. Parts (1) and (2) are now straightforward.
\end{proof}

As in the case of a $\Lambda$-allocation, for any $\alpha\in\Lambda$ we have a function $\EuScript{R}(\alpha)\in\rm{App}(A,\Lambda)$ given by $\EuScript{R}(\alpha)(a,b)=\alpha$. A direct calculation gives:

\begin{prop}
\label{d11}
Let be $\Lambda$ a complete lattice and $A$ an idiom. Then, the function $\EuScript{R}:\Lambda\rightarrow \rm{App}(A,\Lambda)$ defined by $\EuScript{R}(\alpha)(a,b)=\alpha$, is an embedding in the category of complete lattices. 
\end{prop}
\qed

As an example of this, define $\xi:\EuScript{I}(A)\rightarrow \EuScript{D}(A)$ by $\xi(a,b)=\EuScript{D}(a,b)$, where $\EuScript{D}(a,b)$ is the least division set that contains the interval $[a,b]$. It is clear that $\xi(\_)$ is a $\EuScript{D}(A)$-aspect of $A$.

\begin{thm}
\label{d12}
Let $A$ any idiom and $\Lambda$ a complete lattice. Then, there is poset-morphism 
 $$\mathcal{H}: \rm{App}(A,\Lambda)\longrightarrow \rm{Sit}(A,\Lambda)^{op}$$ 
 given for $\psi\in\rm{App}(A,\Lambda)$ by $\mathcal{H}(\psi)(a,b)=\bigvee\left\{\alpha\in\Lambda|[a,b]\in\EuScript{D}vs(\EuScript{M}(\psi,\alpha))\right\}$.
\end{thm}

\begin{proof}
We must check that $\mathcal{H}(\psi)\in\rm{Sit}(A,\Lambda)$. First observe that condition (1) of Definition \ref{d1} is clearly satisfied. 
For the other parts of Definition \ref{d1}, consider any  interval $[a,c]$ on $A$ and take $a\leq b\leq c$. From the definition of $\mathcal{H}(\psi)$ we have that $\mathcal{H}(\psi)(a,c)\leq\mathcal{H}(\psi)(a,b)$. Now, 
from the fact that $\psi(a,b)\vee\psi(b,c)=\psi(a,c)$ we deduce that $\mathcal{H}(\psi)(a,b)\wedge\mathcal{H}(\psi)(b,c)\leq \mathcal{H}(\psi)(a,c)$. Now, let $X\subseteq [a,\bar{1}]$ be a directed set. From the above paragraph we have that $\mathcal{H}(\psi)(a,\bigvee X)\leq\bigwedge\left\{\mathcal{H}(\psi)(a,x)|x\in X\right\}$.  For the other comparison, observe that since $\psi$ is a $\Lambda$-aspect, we have $\psi(a,\bigvee X)=\bigvee\left\{\psi(a,x)|x\in X\right\}$.  Thus, $\psi(a,x)\leq\psi(a,\bigvee X)$. Hence, $\psi(a,\bigvee X)\leq\mathcal{H}(\psi)(a,x)$ for all $x\in X$, and therefore $\psi(a,\bigvee X)\leq \bigwedge\left\{\mathcal{H}(\psi)(a,x)|x\in X\right\}$. Thus, $\bigwedge\left\{\mathcal{H}(\psi)(a,x)|x\in X\right\}\leq\mathcal{H}(\psi)(a,\bigvee X)$.
Lastly, consider any $\psi\leq\psi'$ in $\rm{App}(A,\Lambda)$. Then, from the definition of $\mathcal{H}$ we have that $\mathcal{H}(\psi')\leq\mathcal{H}(\psi)$, that is, $\mathcal{H}$ is a poset morphism.
\end{proof}
%%%%%%%%%%%%%%%%%%%%%%%%%%%%%%%%%%%%%%%%%%%%%%%%%%%%%

\section{Some constructions for $\rm{Sit}(A,\Gamma)$}\label{sec:sec4}

In this section we analyse how the elements of $\rm{Sit}(A,\Gamma)$ lead to decomposition theories for the idiom $A$.  
As in the case of categories of modules, the concept of radical function in idioms is natural in this context (see, for example \cite{7} and \cite{22}): 

\begin{dfn}
\label{dct1}
Let be $A$ an idiom and $\Omega$ a partial ordered set. A  function $\rho:\EuScript{I}(A)\longrightarrow \Omega$ is a \emph{radical function} if:
\begin{enumerate}
\item $\rho(l\wedge r,r)=\rho(l,l\vee r)$ for any $r,l\in A$, and 
\item $\rho(a,c)\leq\rho(a,b)$ for all $a\leq b\leq c$ .
\end{enumerate} 
\end{dfn}
Examples of these functions are, of course, any $\Gamma$-allocation for $A$. In particular, the $N(A)$-allocation $\chi$. 
Other examples of these functions come from module theory: Denote by $\Lambda(M)$ the idiom of sub-modules of an $R$-module $M$, and consider any radical function in $R$-$\Mod$ in the sense of \cite[Chapter 3]{7}. Then, the restriction of any radical function to the intervals of $\Lambda(M)$ gives a radical function in our sense.
If $\Omega$ is a lattice and $\rho\in\rm{Rad}(A,\Omega)$, for $l \in\Omega$ define $\rho':\EuScript{I}(A)\rightarrow \Omega$ by $\rho'(a,b)=\rho(a,b)\wedge l$. Then, $\rho'$ is a radical function on $A$ with values in the lattice $\Omega$. 

Let $\rm{Rad}(A,\Omega)=\left\{\rho:\EuScript{I}(A)\rightarrow \Omega\;|\; \rho\text{ is  radical }\right\}$. We define a partial order in $\rm{Rad}(A,\Omega)$ using the order of $\Omega$ as follows: $\rho\leq\varrho\Leftrightarrow \rho(a,b)\leq\varrho(a,b)$ for all $[a,b]\in\EuScript{I}(A)$. 
As in the case of $\rm{Sit}(A,\Gamma)$, we have that $\rm{Rad}(A,\_):\EuScript{P}os\longrightarrow \EuScript{P}os$ is a covariant functor and  $\rm{Rad}(\_,\Omega):\EuScript{I}\EuScript{D}\longrightarrow \EuScript{P}os$ is a contravariant functor. If $\Gamma$ is a complete lattice, then $\rm{Sit}(A,\Gamma)$ is included in $\rm{Rad}(A,\Gamma)$. 

\begin{dfn}
\label{dct2}
Let  $\rho\in\rm{Rad}(A,\Omega)$.  An interval $[a,b]$ is $\rho$-\emph{stable} or $\rho$-\emph{inert} if and only if $a<b$ and $\rho(a,b)=\rho(a,x)$ for all $a<x\leq b$.
\end{dfn}

For $\chi\in\rm{Sit}(A,N(A))$, the $\chi$-stable intervals are precisely the inert intervals. In particular, any uniform interval is inert, \cite{25}. In fact, in \cite{25} the author describes in detail the decomposition theory generated by $\chi$, and gives an application to geo-lattices. 
Now, for $\rho\in\rm{Rad}(A,\Omega)$, the \emph{support} of $\rho$ is the set  
$$\Sigma_{\rho}(a,b)=\left\{\rho(a,x)\;|\;a<x\leq b\text{ is }\rho\text{-stable}\right\}.$$

\begin{prop}
\label{dct3}
The function $\Sigma_{\rho}:\EuScript{I}(A)\rightarrow \mathcal{P}(\Omega)^{\text{\rm op}}$ which assigns to each interval $[a,b]$ the support of the radical function $\rho$, is a $\mathcal{P}(\Omega)^{\text{\rm op}}$-allocation.
\end{prop}

\begin{proof}
Let  $\rho\in\rm{Rad}(A,\Omega)$. By definition of radical function the first requirement to be a allocation is clearly satisfied, that is, $\Sigma_{\rho}(r\wedge l,r)=\Sigma_{p}(l,r\vee l)$ for any $l,r\in A$.
Now, consider any interval $[a,c]$ and $a\leq b \leq c$. Then $\Sigma_{\rho}(a,b)\subseteq \Sigma_{\rho}(a,c)$.
Take any $\rho(a,x)\in\Sigma_{\rho}(a,c)$. If $a<b\wedge x$, then $\rho(a,x)=\rho(a,b\wedge x)\in\Sigma_{\rho}(a,b)$, and if $a=b\wedge x$ we have $\rho(a,x)=\rho(b\wedge x,x)=\rho(b,b\vee x)$. Thus, from the above this last interval is $\rho$-stable and then $\rho(a,x)\in\Sigma_{\rho}(b,c)$, that is, $\Sigma_{\rho}(a,c)\subseteq \Sigma_{\rho}(a,b)\cup\Sigma_{\rho}(b,c)$.
Now, for the last requirement take any $X\subseteq [a,\bar{1}]$ directed, for some $a\in A$ and consider $\rho(a,y)\in\Sigma_{\rho}(a,\bigvee X)$. Thus, from the idiom distributivity law we have $y=y\wedge (\bigvee X)=\bigvee\left\{y\wedge x\;|\; x\in X\right\}$. Then, $a<y\wedge x\leq x$ for some $x\in X$, and from $a<y\wedge x\leq y$ we derive $\rho(a,y)=\rho(a,y\wedge x)\in\Sigma_{\rho}(a,x)$. This proves that $\Sigma_{\rho}(a,\bigvee X)\subseteq\bigcup\left\{\Sigma_{\rho}(a,x)\;|\;x\in X\right\}$.
The other comparison follows immediately from the first property.
\end{proof}

Assume now that $\Omega$ is a complete lattice and consider $\varphi\in\rm{Sit}(A,\mathcal{P}(\Omega)^{op})$. Define the function $\varrho_{\varphi}:\EuScript{I}(A)\rightarrow \Omega$ by $\varrho_{\varphi}(a,b)=\bigwedge\varphi(a,b)$. This function is clearly a radical function. Thus we have two functions $$\xymatrix{ \rm{Sit}(A,\mathcal{P}(\Omega)^{op})\ar@<-.7ex> @/   _1pc/[r]_--{\varrho} &\rm{Rad}(A,\Omega)\ar@<-.7ex>@/ _1pc/[l]_--{\Sigma}       }$$ where $\Sigma(\rho)=\Sigma_{\rho}$ and $\varrho(\varphi)=\varrho_{\varphi}$. Observe now that we have a diagram
$$\xymatrix{ {\rm{Sit}(A,\Omega)}  \ar[rr]^{\iota}   \ar[drr]_{\left\{\_\right\}^{*}}  &&{\rm{Rad}(A,\Omega)}  \ar@<-.7ex>  @/   _1pc/[d]_--{\Sigma} \\   && {\rm{Sit}(A,\mathcal{P}(\Omega)^{op})}   \ar@<-.7ex>@/ _1pc/[u]_--{\varrho} } $$
 where $\left\{\_\right\}^{*}:\rm{Sit}(A,\Omega)\rightarrow \rm{Sit}(A,\mathcal{P}(\Omega)^{op})$ is the  morphism induced by the inclusion  $\Omega \rightarrow \mathcal{P}(\Omega)^{op}$. Note that $\varrho\circ\left\{\_\right\}^{*}=\iota$ and so one triangle is commutative, but the other one does not need to commute.

\begin{dfn}
\label{dct4}
For $\rho\in\rm{Rad}(A,\Omega)$, an interval $[a,b]$ is $\rho$-\emph{atomic} if $\Sigma_{\rho}(a,b)=\left\{*\right\}$. 
For $\rho\in\rm{Rad}(A,\Omega)$, the idiom $A$ is $\rho$-\emph{adequate} if $\Sigma_{\rho}(a,b)$ is not empty for every non-trivial interval $[a,b]$ of $A$. 
For  $\rho\in\rm{Rad}(A,\Omega)$ and a element $p\in \Omega$, an interval $[a,b]$ is $p$\emph{-inertial} if $\rho(a,b)=p$ and $[a,b]$ is $\rho$-stable.
\end{dfn}

From this observe that any $\rho$-stable interval $[a,b]$ is a $p=\rho(a,b)$-inertial. We will use $p$-inertial intervals to  generate a decomposition for the parent idiom. 

Using intervals $p$-inert  with respect to some $\varphi\in\rm{Sit}(A,\Omega)$ give us another look at allocations. For $p\in\Omega$, consider the frame $\large 2$ with two elements $0<1$, and define the function $p:\EuScript{I}(A)\longrightarrow \large 2$ by 
\begin{displaymath}
p(a,b)=\left\{\begin{array}{ll}
1 & \textrm{if $[a,b]$ is $p$-inertial} \\
0 & \textrm{if not}, 
\end{array}\right.
\end{displaymath}
 for each interval $[a,b]$.
\begin{prop}
\label{dtc5}
For each $p\in\Omega$, the function $p:\EuScript{I}(A)\longrightarrow \large 2$ is a $\large 2$-allocation, that is: 

\begin{itemize}
\item[(1)] For $l,r\in A$, $[l\wedge r,r]$ is $p$-inert $\Leftrightarrow [l,l\vee r]$ is $p$-inert.
\item[(2)] For $a\leq b\leq c$, $[a,c]$  $p$-inert $\Rightarrow$ $[a,b]$ is $p$-inert.
\item[(3)] For $a\leq b\leq c$, $[a,b]$ and $[b,c]$ $p$-inert $\Rightarrow$ $[a,c]$ is $p$-inert.
\item[(4)] For $a\in A$ and $X\subseteq [a,1]$, $[a,\bigvee X]$  is $p$-inert $\Leftrightarrow$ $(\forall x\in X)\left[[a,x] \text{ is } p\text{-inert}\right]$.
\end{itemize}
\end{prop}

\begin{proof}
(1): For any $l,r\in A$, first suppose that $[l\wedge r,r]$ is $p$-inert and consider $l<x\leq l\vee r$. Then, using the canonical isomorphism $[l\wedge r,r]\cong [l, l\vee r]$ we have that $x=y\vee l$ for some $y<l\wedge r\leq r$. Therefore, $\varphi(l,x)=\varphi(l,y\vee l)=\varphi(l\wedge r, y)=\varphi(l\wedge r,r)=p$, where the second equality is because the axioms of allocations and the third one is by the hypothesis.
The reverse implication is similar.

(2): Let  $a\leq b\leq c$ with $[a,c]$ $p=\varphi(a,c)$-inert. Then, for any $a<x\leq b$ we have $\varphi(a,x)=\varphi(a,c)=\varphi(a,b)$.

(3): Given $[a,b]$ and $[b,c]$ $p$-inert intervals and any $a<x\leq c$ we have that $p=\varphi(a,b)\wedge \varphi(b,c)\leq \varphi(a,c)\leq\varphi(a,x)$. Thus we only need to show that $\varphi(a,x)\leq p$.  First observe that $a\leq b\wedge x\leq b$. Now,  if $a=b\wedge x$ we have $b<b\vee x\leq c$, and therefore $\varphi(a,x)=\varphi(b\wedge x,x)=\varphi(b,b\vee x)=\varphi(b,c)=p$, where  the second equality comes from the fact that $[b,b\vee x]\cong [b\wedge x,x]$.
 
(4): Let $a\in A$ and $X\subseteq [a,\bar{1}]$ directed. It is enough to verify that $(\forall x\in X)\left[[a,x] \text{ is } p\text{-inert}\right]$ implies $[a,\bigvee X]$ is $p$-inert. To see this, note that since$X$ is directed we have that $\varphi(a,\bigvee X)=\bigwedge\left\{\varphi(a,x)\;|\;x\in X\right\}=p$. Consider  now $a<y\leq \bigvee X$.  Then, $p=\varphi(a,\bigvee X)\leq \varphi(a,y)$. To show the other comparison note that $y=y\wedge \left(\bigvee X\right)=\bigvee\left\{y\wedge x\;|\; x\in X\right\}$ using the idiom distributivity law. Then, for some $x\in X$ we have that $a\leq y\wedge x\leq y$ and $a< y\wedge x\leq x$. It follows that $\varphi(a,y)\leq\varphi(a,y\wedge x)=\varphi(a,x)=p$.
\end{proof}

Now, for $\varphi\in\rm{Sit}(A,\Omega)$, let $\EuScript{D}_{p}=\left\{[a,b]\;|\; [a,b]\text{ is  $p$-inert}\right\}$. The last proposition says that $\EuScript{D}_{p}$ is a congruence set in $A$. Moreover, we know that for any congruence set $\mathcal{C}$, if $[a,x],[a,y]\in\mathcal{C}$ then $[a,x\vee y],[a,x\wedge y]\in\mathcal{C}$. From this it is easily seen that $\mathcal{C}$ is closed under finite suprema.

\begin{cor}
\label{dtc5.}
The set $\EuScript{D}_{p}$ is a division set in $A$.
\end{cor}

\begin{proof}
Take any $a\in A$ and $X\subset [a,\bar{1}]$ with $[a,x]$ $p$-inert for all $x\in X$.  Let  $Y$ be the set of all elements of the form $x_{1}\vee x_{2}\ldots\vee x_{n}$, with $x_{i}\in X$ for $0\leq i\leq n$. This  is a directed set and $[a,y]$ is $p$-inert.  Using the same reasoning in the  proof of (4) in Proposition \ref{dtc5}, we have $p=\varphi(a,\bigvee Y)=\bigwedge\left\{\varphi(a,y)\;|\;y\in Y\right\}\leq\varphi(a,\bigvee X)\leq\bigwedge\left\{\varphi(a,x)\;|\; x\in X\right\}=p$, and for any $a<z\leq \bigvee X\leq \bigvee Y$ there is some $y\in Y$ with $a<z\wedge y\leq y$ then $\varphi(a,z)\leq\varphi(a,z\wedge y)=\varphi(a,y)=p$. The other comparison is clear.  
\end{proof}

\begin{dfn}
\label{dtc6}
Let be $[a,b]$ an interval over an idiom $A$. An element $a\leq x\leq b$ is a $p$\emph{-inertial point} or a $p$\emph{-stable point} in $[a,b]$, with $p\in\Omega$, if $[a,x]$ is $p$-inertial and if $x\wedge y=a$ then $[a,y]$ is not $p$-inertial for each $a\leq y\leq b$. The element $x$ is an \emph{inertial point} or  a \emph{stable point} in $[a,b]$, if it is a $p$-inertial point for some $p\in\Omega$.
\end{dfn}

For the remaining part of this  section we use concepts and results on independent sets on idioms as  in \cite{22}.
We start by showing that there are  enough inertial points in an idiom:

\begin{prop}
\label{dtc7}
Let $[a,b]$ be an interval of an idiom $A$, and $\rho\in\rm{Sit}(A,\Omega)$ with $\Omega$ a complete lattice and consider $p\in\Sigma_{\rho}(a,b)$. Then, for each $a\leq z\leq b$ with $[a,z]$ $p$-inertial there is a $p$-inertial point $z\leq x\leq b$ in $[a,b]$.
\end{prop}

\begin{proof}
We use Zorn's lemma: Consider the family $\Pi$ of subsets $X\subseteq [a,b]$ satisfying:
\begin{enumerate}
\item $z\in X$ .
\item $X$ is independent over $a$.
\item For each $x\in X$ the interval $[a,x]$ is $p$-inertial.
\end{enumerate}
By hypothesis,  $z$ is an inertial point and thus gives an element $\left\{z\right\}$ in $\Pi$. Inclusion is a partial order in
 $\Pi$. Consider any chain $\mathcal{Z}$ of elements of $\Pi$,  and its union $\bigcup\mathcal{Z}$. Clearly, $\bigcup\mathcal{Z}\in\Pi$ and by Zorn's lemma there exists a maximal member $X$ of $\Pi$.  If $x=\bigvee X$, then $a\leq x\leq b$ and by Proposition \ref{dtc5} it follows that $[a,x]$ is $p$-inertial.

Lastly, consider $a\leq y\leq b$ with $x\wedge y=a$. Then, the family $X\cup\left\{y\right\}$ is independent over $a$ and the maximality of $X$ implies that $[a,y]$ is not $p$-inertial.
\end{proof}

\begin{lem}
\label{dtc8}
Let be $A$ an idiom and $\varphi\in\rm{Sit}(A,\Omega)$. Suppose that $A$ is $\varphi$-adequate. Then, for each interval $[a,b]$ we have that 
$$\chi(a,b)=\bigwedge\left\{\chi(a,x)\;|\;a<x\leq b\text{ with $[a,x]$ }\varphi\text{-stable}\right\}.$$
\end{lem}

\begin{proof}
Since $\chi$ is a $N(A)$-allocation,   
$$\chi(a,b)\leq\bigwedge\left\{\chi(a,x)\;|\;a<x\leq b\text{ with $[a,x]$ }\varphi\text{-stable}\right\}.$$
 For the other comparison, let  $\Xi=\left\{\chi(a,x)\;|\;a<x\leq b\text{ is }\varphi\text{-stable}\right\}$ and $k=\bigwedge\Xi$. If $a<k(a)\wedge b$, by  hypothesis 
there is $a<x\leq k(a)\wedge b$ with $[a,x]$ $\varphi$-stable. Then,  $\chi(a,x)\in\Xi$ and thus $k\leq\chi(a,x)$. Hence,  $x\leq k(a)\leq\chi(a,x)(a)$ and then $x=\chi(a,x)(a)\wedge x=a$, which is a contradiction. 
\end{proof}

The concept of $p$-inertial point is related to the concept of a large element:

\begin{lem}
\label{dtc9}
Let be $A$ an idiom, $\varphi$-adequate for some $\varphi\in\rm{Sit}(A,\Omega)$ and suppose $[a,b]$ is $\varphi$-atomic, that is, $\Sigma_{\varphi}(a,b)=\left\{p\right\}$. Then, any $p$-inertial point in $[a,b]$ is large in $[a,b]$.
\end{lem}

\begin{proof}
Suppose $x$ is a $p$-inertial point in $[a,b]$.  Then, $[a,x]$ is $p$-inert in $[a,b]$. Consider any $y\in [a,b]$ with $a=x\wedge y$ and suppose $a<y$. Since $A$ is  $\varphi$-adequate, there is some $a<z\leq y$ with $[a,z]$ $\varphi$-stable. Then $a\leq z\wedge x\leq x\wedge y=a$, which contradicts the $p$-point property of $x$.
\end{proof}

We can now extend the definition of decomposition for a interval $[a,b]$ over an idiom $A$.

\begin{dfn}
\label{dtc10}
Let $A$ be an idiom and $\varphi\in\rm{Sit}(A,\Omega)$. A $\varphi$\emph{-decomposition} of an interval $[a,b]$ of $A$ is a family $X=\left\{x_{p}\:|\; p\in\Sigma_{\varphi}(a,b)\right\}$ of elements of $[a,b]$ indexed by the support of $\varphi$ such that:

\begin{itemize}
\item[(1)] $X$ is independent over $a$.
\item[(2)] $\bigvee X$ is large in $[a,b]$.
\item[(3)] The interval $[a,x_{p}]$ is $p$-inert for each $p\in\Sigma_{\varphi}(a,b)$.
\end{itemize}
\end{dfn}
\begin{thm}
\label{dtc11}
For an idiom $A$ and $\varphi$ a $\Omega$-allocation, the following are equivalent:
\begin{itemize}
\item[(1)] Each non-trivial interval of $A$ has a $\varphi$-decomposition.
\item[(2)] $A$ is $\varphi$-adequate.
\end{itemize}
\end{thm}

\begin{proof}
Assuming $(1)$, every non-trivial interval $[a,b]$ in $A$ has a $\varphi$-decomposition  $X$ of $[a,b]$, with $\bigvee X$ large in $[a,b]$.  Then, this element is not $a$ and so $X$ is not empty. Thus, $\Sigma_{\varphi}(a,b)$ is not empty.

Now assume $(2)$ and consider any non-trivial interval $[a,b]$ of $A$. By Proposition \ref{dtc7} there is a family $X=\left\{x_{p}\;|\; p\in\Sigma_{\varphi}(a,b)\right\}\subseteq [a,b]$ such that $x_{p}$ is a $p$-inertial point in $[a,b]$, and $[a,x_{p}]$ are $p$-inert intervals. To verify parts (1) and (2) of Definition \ref{dtc10} it is enough to prove that $X$ is independent over $a$. Thus, we only need to check that every finite subset of $X$ is independent over $a$. Let $Y$ be a finite subset of $X$. We do induction on the cardinality of $Y$.  Consider $p,p_{1}\dots,p_{n}$ distinct elements of $\Sigma_{\varphi}(a,b)$ such that $Y=\left\{p,p_{1}\dots,p_{n}\right\}$. By  induction hypothesis we know that ${x_{p_{1}}\dots,x_{p_{n}}}$ are independent over $a$. To show the independence of $Y\cup{x_{p}}$ over $a$, let $y=\bigvee Y$. Then, $\Sigma_{\varphi}(a,y)=\bigwedge\left\{p_{1},\ldots,p_{n}\right\}$ because $\Sigma_{\varphi}$ is a $\mathcal{P}(\Omega)$-allocation, and we also have that $\Sigma_{\varphi}(a,x_{p})={p}$.  Then, $\Sigma_{\varphi}(a,y\wedge x_{p})=\Sigma_{\varphi}(a,y)\cap\Sigma_{\varphi}(a,x_{p})=\emptyset$ by $(2)$ of Definition \ref{d1}.  Since $A$ is $\varphi$-adequate, then $x\wedge x_{p}=a$  and thus $Y\cup{x_{p}}$ is independent over $a$. To verify $(2)$ of  Definition \ref{dtc10} suppose that $x=\bigvee X$ is not large in $[a,b]$, that is, there exists $a<y\leq b$ with $x\wedge y=a$.  By the hypothesis (2) we can assume that $[a,y]$ is $\varphi$-inert with $\varphi(a,y)=p\in\Sigma_{\varphi}(a,b)$. Thus, the element $x_{p}$ is a $p$-inertial point in $[a,b]$ and $x_{p}\wedge y\leq x\wedge y=a$, which is a contradiction. 
\end{proof}

Theorem \ref{dtc11} is a bit more general than Theorem $8.2$ in \cite{25} which is the special case of an $NA$-allocation $\chi$. In  \cite{25} the author applies this to geo-lattices and the decomposition theory generated by $\chi$ in connection with certain spatial properties of the corresponding idiom, that is, any $\chi$-stable interval $[a,b]$ gives a point (a $\wedge$-irreducible element) of $N(A)$. Thus, the resulting decomposition theory has a more module theoretic flavour. 
%%%%%%%%%%%%%%%%
\section{Some constructions in $\rm{App}(A,\Lambda)$}\label{sec:sec5}

The concept of dimension in an idiom can be stated in many ways, depending of the context, for example via inflators and nuclei as in \cite{28}. In this section we will give the lattice theoretical constructions of these via $\rm{App}(A,\Lambda)$. these constructions are the idiomatic version of the one developed  in \cite{7}. 

Let $\Lambda$ be a complete lattice with $\top,\bot$ his top and his bottom elements, respectively.  Denote by $\propto(\Lambda)$ the minimum of all cardinals $\iota$ such that $\iota> \#\Lambda$.  Let $\infty(\Lambda)=\left\{\kappa \;|\;\kappa \text{ is an ordinal and } \kappa\leq \propto(\Lambda)\right\}$.
Define 
$$\text{\rm seq}(\Lambda)=\left\{h:\infty(\Lambda)\rightarrow \Lambda\; |\; h\; \text{\rm is  increasing and }h(0)=\bot,\;\; h(\propto(\Lambda))=\top  \right\}.$$
 Note that $\rm{seq}(\Lambda)$ is a complete lattice in the usual way. Now, let $h\in\rm{seq}(A)$;  then, there is an ordinal $\alpha<\propto(\Lambda)$ such that $h(\alpha)=h(\alpha+1)=\cdots$. The least of these ordinals will be denoted by 
 $\rm{Bnd}(h)$.

Let be $A$ an idiom and $\Lambda$ a complete lattice, for $\psi\in\rm{App}(A,\Lambda)$ and $h\in\rm{seq}(\Lambda)$. We dine $\mathfrak{d}^{\psi}_{h}: \EuScript{I}(A)\rightarrow \infty(\Lambda)$ inductively as follows:
\medskip

\noindent{\rm (i)} $\mathfrak{d}^{\psi}_{h}(a)=0$ for all $a\in A$.
\medskip

\noindent{\rm (ii)} If $[a,b]\in\EuScript{I}(A)$ is non trivial, then 
$$\mathfrak{d}^{\psi}_{h}(a,b)=\inf\{0\leq \iota\leq \propto(\Lambda)\;|\; \psi(a,b)\leq h(\iota)\}.$$

\begin{prop}
\label{d13}
Let be $A$ an idiom and $\Lambda$ a complete lattice. Then, the function $\mathfrak{d}^{\psi}_{h}$ is a $\infty(\Lambda)$-aspect for each $\psi\in\rm{App}(A,\Lambda)$ and $h\in\rm{seq}(\Lambda)$.
Moreover:
\begin{itemize}
\item[(1)] For every $\psi\in\rm{App}(A,\Lambda)$, the function $\mathfrak{d}^{\psi}:\rm{seq}(\Lambda)^{op}\rightarrow \rm{App}(A,\infty(\Lambda))$ is a $\bigvee$-morphism.

\item[(2)] For every $h\in \rm{seq}(\Lambda)$, the function $\mathfrak{d}_{h}:\rm{App}(A,\Lambda)\rightarrow \rm{App}(A,\infty(\Lambda))$  is a $\bigvee$-morphism.
\end{itemize}
\end{prop}

\begin{proof}
Let $\psi\in\rm{App}(A,\Lambda)$ and $h\in \rm{seq}(\Lambda)$. The first requirement of Definition \ref{d7} is clearly satisfied. Now consider any  non-trivial interval $[a,c]$ on $A$ and take $a\leq b\leq c$.  Then, $\psi(a,c)=\psi(a,b)\vee\psi(b,c)$ and so $\mathfrak{d}^{\psi}_{h}(a,c)=\sup\{\mathfrak{d}^{\psi}_{h}(a,b),\mathfrak{d}^{\psi}_{h}(b,c)\}$. If $X$ is a directed subset of $[a,\bar{1}]$, then $\psi(a,\bigvee X)=\bigvee\{\psi(a,x)\;|\; x\in X\}$. Therefore $\mathfrak{d}^{\psi}_{h}(a,\bigvee X)=\sup\{\mathfrak{d}^{\psi}_{h}(a,x)\;|\; x\in X\}$.

Let $H=\left\{h_{j}\;|\; j\in J\right\}$ be a family in $\rm{seq}(\Lambda)$, and $h=\bigwedge H$. Consider any interval $[a,b]$ on $A$, and let $\iota=\mathfrak{d}^{\psi}_{h}(a,b)$,  $B(j)=\mathfrak{d}^{\psi}_{h_{j}}(a,b)$, and $B=\sup\left\{B(j)\;|\; j\in J\right\}$, for each $j\in J$. Then, $\psi(a,b)\leq h(\iota)\leq h_{j}(\iota)$ and thus $\iota\geq B(j)$, for each $j\in J$.  Hence $\iota\geq B$.
For the other comparison, we have $\psi(a,b)\leq h_{j}(B)$ for each $j \in J$. Then, $\psi(a,b)\leq h(B)$ and thus $\iota\leq B$, that is, $\iota=B$. This proves assertion (1).

Next, consider $\psi=\bigvee\left\{\psi_{j}\;|\; j\in J\right\}$ in $\rm{App}(A,\Lambda)$ and $[a,b]$ an interval over $A$. If $\iota=\mathfrak{d}^{\psi}_{h}(a,b)$, $B(j)=\mathfrak{d}^{\psi_{j}}_{h}(a,b)$, and $B=\sup\left\{B(j)\;|\; j\in J\right\}$, for each $j\in J$, then $\psi_{j}(a,b)\leq h(B(j))\leq h(B)$ for each $j\in J$.  Thus, $\psi(a,b)\leq h(B)$ and so $\iota\leq B$.  Now, if this inequality is strict, then there exists a $j\in J$ such that $B(j)>\iota$. Hence, $\psi_{j}(a,b)\nleq h(\iota)$ and so $\psi(a,b)\nleq h(\iota)$, which is a contradiction. So we must have $\iota=B$, and this proves assertion (2). 
\end{proof}

Note that for every $\alpha\in\Lambda$ we have an embedding from $\Lambda$ into $\rm{seq}(\Lambda)$ given by $\alpha\mapsto h^{\alpha}$, where $h^{\alpha}(\iota)=\alpha$. With this and the definitions above we obtain:

\begin{cor}
\label{d14}
Let be $A$ and idiom and $\Lambda$ a complete lattice. Then, for any element $\alpha\in\Lambda$ the diagram: $$\xymatrix{ \rm{App}(A,\Lambda)^{op}\ar[rr]^{\mathfrak{d}_{h^{\alpha}}}\ar[dr]_{\EuScript{M}(\_,\alpha)} && \rm{App}(A,\infty(\Lambda))\ar[dl]^{\EuScript{M}(\_,0)}\\ & \EuScript{C}(A)  }$$ conmutes.
\qed
\end{cor}

The method developed in Proposition \ref{d13} is the idiomatic version of the one described in \cite{7}. Now remember from Section \ref{sec:sec2}, that we can construct certain operations over the base frame $\EuScript{B}(A)$ for any idiom $A$, that is, certain inflators $\EuScript{O}pr:\EuScript{B}(A)\rightarrow \EuScript{B}(A)$. With these we can define sequences $h_{\psi,\alpha}\in\rm{seq}(\Lambda)$, for each $\psi\in\rm{App}(A,\Lambda)$ and $\alpha\in\Lambda$, as follows:

\begin{itemize}
\item[(1)] $h_{\psi,\alpha}(0)=\alpha$.

\item[(2)] If $0< \iota<\propto(\Lambda)$, then 
$$h_{\psi,\alpha}(\iota)=h_{\psi,\alpha}(\iota-1)\vee\Big(\bigvee\left\{\psi(a,b)\;|\; [a,b]\in \EuScript{O}pr(\EuScript{M}(\psi,h_{\psi,\alpha}(\iota-1)))\right\}\Big).$$

\item[(3)] If $0< \iota<\propto(\Lambda)$ is a limit ordinal, then $h_{\psi,\alpha}(\iota)=\bigvee\left\{h_{\psi,\alpha}(\lambda)\;|\; \lambda< \iota\right\}$.
\end{itemize}

This sequence will be called the $\EuScript{O}pr$-\emph{filtration}. From this we can apply $\mathfrak{d}$ to $\psi$ and the $\EuScript{O}pr$-\emph{filtration} to obtain $\mathfrak{d}^{\psi}_{h_{\psi,\alpha}}\in \rm{App}(A,\infty(\Lambda))$. This aspect will be called the $(\psi,\alpha)$-\emph{dimension}, or $(\psi,\alpha)$-dim for short.

Now consider the particular case when $\Lambda=\EuScript{D}(A)$, and $A$ is an idiom. Then, take the $\EuScript{D}(A)$-aspect $\xi(\_)$  and the inflator $\EuScript{D}vs\circ \EuScript{O}pr:= \EuScript{K}pr$, and define a $(\mathcal{D},\EuScript{K}pr)$-filtration of any $\mathcal{D}\in \EuScript{D}(A)$ as follows:
\begin{itemize}
\item[(1)] $\EuScript{K}pr^{0}(\mathcal{D})=\mathcal{D}$.

\item[(2)] $\EuScript{K}pr^{\gamma+1}(\mathcal{D})=\EuScript{K}pr(\EuScript{K}pr^{\gamma}(\mathcal{D}))$.

\item[(3)] $\EuScript{K}pr^{\lambda}(\mathcal{D})=\EuScript{D}vs(\bigcup\left\{\EuScript{K}pr^{\beta}(\mathcal{D})\;|\; \beta<\lambda\right\})$.
For each ordinal $\gamma$ and limit ordinal $\lambda$.
\end{itemize}
\begin{prop}
\label{d15}
Let $A$ be an idiom. With the above  notations we have that the $\EuScript{O}pr$-filtration and the $(\mathcal{D},\EuScript{K}pr)$-filtration are the same.
\end{prop}

\begin{proof}
We proceed by induction over the ordinals $\gamma$ and limit ordinals $\lambda$. 
The case $\gamma=0$ is trivial. 
For the induction step $\gamma\mapsto \gamma+1$ observe that by definition of the sequence $$h_{\xi(\_),\mathcal{D}}(\gamma)=h_{\xi(\_),\mathcal{D}}(\gamma-1)\vee\left(\bigvee\left\{\xi(a,b)\;|\; [a,b]\in \EuScript{O}pr(\EuScript{M}(\xi(\_),h_{\xi(\_),\mathcal{D}}(\gamma-1)))\right\}\right)$$ we have that the congruence set $\EuScript{M}(\xi(\_),h_{\xi(\_),\mathcal{D}}(\gamma-1))=\EuScript{M}(\xi(\_),\EuScript{K}pr^{\gamma-1}(\mathcal{D}))$. Recall that $[a,b]\in\EuScript{M}(\xi(\_),h_{\xi(\_),\mathcal{D}}(\gamma-1))$ precisely when $\xi(a,b)\leq h_{\xi(\_),\mathcal{D}}(\gamma-1)$, that is, $\xi(a,b)\subseteq\EuScript{K}pr^{\gamma-1}(\mathcal{D})$.  Thus, $\EuScript{M}(\xi(\_),h_{\xi(\_),\mathcal{D}}(\gamma-1))=\EuScript{K}pr^{\gamma-1}(\mathcal{D})$, and from this and the induction hypothesis we obtain that  
$$h_{\xi(\_),\mathcal{D}}(\gamma)=\EuScript{K}pr^{\gamma-1}(\mathcal{D})\vee\left(\bigvee\left\{\xi(a,b)\;|\; [a,b]\in \EuScript{O}pr(\EuScript{K}pr^{\gamma-1}(\mathcal{D}))\right\}\right)$$
 in $\EuScript{D}(A)$, that is, 
$$\EuScript{D}vs(\EuScript{K}pr^{\gamma-1}(\mathcal{D})\cup\EuScript{D}vs(\bigcup\left\{\xi(a,b)\;|\; [a,b]\in \EuScript{O}pr(\EuScript{K}pr^{\gamma-1}(\mathcal{D}))\right\}))=$$ $$=\EuScript{D}vs(\bigcup\left\{\xi(a,b)\;|\; [a,b]\in \EuScript{O}pr(\EuScript{K}pr^{\gamma-1}(\mathcal{D}))\right\})$$ because $\EuScript{K}pr^{\gamma-1}(\mathcal{D})\subseteq \EuScript{O}pr(\EuScript{K}pr^{\gamma-1}\mathcal{D})$. 

Let $\mathcal{B}=\EuScript{D}vs(\bigcup\left\{\xi(a,b)\;|\; [a,b]\in \EuScript{O}pr(\EuScript{K}pr^{\gamma-1}(\mathcal{D}))\right\})$. By the description of the $\EuScript{D}vs(\_)$ construction in Theorem \ref{000} for any basic set of intervals, for an interval $[a,b]\in\mathcal{B}$, there exists a proper subinterval $[x,y]$ of $[a,b]$ such that $[x,y]\in \bigcup\left\{\xi(a,b)\;|\; [a,b]\in \EuScript{O}pr(\EuScript{K}pr^{\gamma-1}(\mathcal{D}))\right\}$. Thus, $[x,y]\in \xi(a',b')$ for some $[a',b']\in \EuScript{O}pr(\EuScript{K}pr^{\gamma-1}(\mathcal{D}))$.  From this we can find a subinterval of $[a',b']$ isomorphic to a subinterval of $[x,y]$, say $I$, and this interval is in $\EuScript{O}pr(\EuScript{K}pr^{\gamma-1}(\mathcal{D}))$. This is the case when $[a,b]\in \EuScript{D}vs(\EuScript{O}pr(\EuScript{K}pr^{\gamma-1}(\mathcal{D})))$. Therefore $\mathcal{B}\subseteq \EuScript{D}vs(\EuScript{O}pr(\EuScript{K}pr^{\gamma-1}(\mathcal{D})))$.  The other inclusion is clear. Thus, from the definition of the $\mathcal{D}-\EuScript{K}pr$-filtration we conclude that $\EuScript{K}pr^{\gamma+1}(\mathcal{D})=\EuScript{K}pr(\EuScript{K}pr^{\gamma}(\mathcal{D}))=h_{\xi(\_),\mathcal{D}}(\gamma)$.

Now, for the limit case, we have 
\begin{align*}
h_{\xi,\mathcal{D}}(\lambda)&=\bigvee\left\{h_{\xi,\mathcal{D}}(\beta)\;|\; \beta< \gamma\right\}
=\EuScript{D}vs\left(\bigcup\left\{h_{\xi,\mathcal{D}}(\beta)\;|\; \beta< \gamma\right\}\right)\\
&=\EuScript{D}vs\left(\bigcup\left\{\EuScript{K}pr^{\beta}(\mathcal{D})\;|\; \beta< \gamma\right\}\right)=\EuScript{K}pr^{\lambda}(\mathcal{D}),
\end{align*}
where the second equality is by definition of suprema in $\EuScript{D}(A)$ and the induction hypothesis, the third equality is by the definition of the filtration on $\EuScript{K}pr$ with respect to $\mathcal{D}$ in the limit case.
\end{proof}

Recall that for any  basic set $\mathcal{B}$ on $A$, $\EuScript{C}rt(\mathcal{B})$ is the set of intervals such that for all $a\leq x\leq b$ we have $a=x$ or $[x,b]\in\mathcal{B}$.  This is the set of all $\mathcal{B}$-\emph{critical} intervals. Now, denote by $\EuScript{F}ll(\mathcal{B})$ the set of all intervals $[a,b]$ such that, for all $a\leq x\leq b$ there exists $a\leq y\leq b$ with $a=x\wedge y$ and $[x\vee y,b]\in\mathcal{B}$.  This is the set of all $\mathcal{B}$-\emph{full} intervals. Then,  we consider the $\EuScript{C}tr$-\emph{filtration} and the $\EuScript{F}ll$-filtration in $\rm{seq}(A)$.

Recall that the Gabriel derivative is given by $\EuScript{G}ab(\_):=\EuScript{D}vs\circ\EuScript{C}rt$, and the Boyle derivative is $\EuScript{B}oy(\_):=\EuScript{D}vs\circ\EuScript{F}ll$ on $\EuScript{D}(A)$.  Then, we can construct the $\mathcal{D}-\EuScript{G}ab$-filtration and the $\mathcal{D}-\EuScript{B}oy$-filtration, for some division set $\mathcal{D}$ on $A$. With these and Proposition \ref{d15} the following is immediate.

\begin{cor}
\label{d16}
If  $A$ is an idiom, then the $\EuScript{C}rt$-filtration is exactly the $\mathcal{D}$-Gabriel filtration and the $\EuScript{F}ll$-filtration is exactly the $\mathcal{D}$-Boy filtration.
\qed
\end{cor}
%%%%%%%%%%%%%%%%%%%%%%%%%%%%%

\section{ Dimensions in categories of modules }\label{sec:sec6}

In  \cite{7} the following framework is introduced  to deal with most dimensions in module theory. Recall some of that material. Fix a complete lattice $\Gamma$ and consider a ring $R$, and the category of left modules $R$-$\Mod$.

\begin{dfn}
\label{d17}
A \emph{quasi-dimension function} in $R$-$\Mod$ is a function $R\text{\rm -Mod}\stackrel{D}\longrightarrow \Gamma$ that satisfies:
\begin{enumerate}
\item If $0\rightarrow N\rightarrow M \rightarrow K\rightarrow 0$ is a exact sequence in $R$-$\Mod$ then
$D(M)=D(N)\vee D(K)$.
\item
If $M$ is a module which is a directed union of a directed family $\{N_{i}\;|\; i\in \Omega\}$ of submodules of $M$, then $D(M)=\bigvee\left\{D(N_{i})\;|\; i\in \Omega\right\}$.
\end{enumerate}

If $D(M)=\bot$ $\Leftrightarrow$ $M=0$,  we say that $D$ is of \emph{pre-dimension}. If the image of $D$ is  linearly ordered, we say that $D$ is \emph{linear}. A linear pre-dimension function will be called a \emph{dimension function}.
\end{dfn}

Let $Q$-$\rm{dim}(R,\Gamma)$ be the collection of all quasi-dimension functions in $R$-$\Mod$ with values in $\Gamma$. Let $R$-$\rm{mod}$ be the set of isomorphism classes of finitely generated modules.  It is easily seen that any quasi-dimension function is completely determined by its values in $R$-$\rm{mod}$.  Thus, $Q$-$\rm{dim}(R,\Gamma)$ is a set, and, in fact, a complete lattice.

Denote by $\Lambda(M)$ the idiom of sub-modules of an $R$-module $M$. 

Observe that any $D\in Q\text{-}\dim(R,\Gamma)$ defines a $\Gamma$-aspect of each module as follows: Take any module $M$ and let  $D_{M}:\EuScript{I}(\Lambda(M))\longrightarrow\Gamma$ be defined by $D_{M}(K,L)=D(L/K)$. The definition of quasi-dimension function implies that the function $D_{M}$ is a $\Gamma$-aspect for $\Lambda(M)$. In particular, $D_{R}$ defines a $\Gamma$-aspect for $\Lambda(R)$.

\begin{lem}
\label{d18}
For each $R$-module $M$ we have a morphism of complete lattices
 $$[M]: Q\text{-}\dim(R,\Gamma)\longrightarrow\rm{App}(M,\Gamma)$$ given by $[M](D)=D_{M}$. 
\end{lem}

\begin{proof}
Let  $\mathfrak{D}\subseteq Q$-$\dim(R,\Gamma)$ be a family of quasi-dimensions functions. Then, 
$[M](\bigvee\mathfrak{D})=\bigvee\mathfrak{D}_{M}$, and thus 
\begin{align*}
\big(\bigvee\mathfrak{D}_{M}\big)(K,L)&=\big(\bigvee\mathfrak{D}\big)(L/K)=\bigvee\left\{D(L/K)\;|\;D\in\mathfrak{D}\right\}\\
&=\bigvee\left\{D_{M}(K,L)\;|\;D\in\mathfrak{D}\right\}=\bigvee\left\{[M](D)\;|\;D\in\mathfrak{D}\right\},
\end{align*}
as required. For $[M](\bigwedge\mathfrak{D})=\bigwedge\left\{[M](D)\;|\;D\in\mathfrak{D}\right\}$ the proof is similar.
\end{proof}

Denote by $\mathbb{B}(R)$ the collection of all classes of modules closed under isomorphism, sub-modules and quotients.  This is the \emph{base frame} of $R$-$\Mod$ thus as in \cite{23} we can do the idiomatic constructions of $\mathbb{C}(R)$ and $\mathbb{D}(R)$. Note that this frames are the classes of all Serre classes and the frame of all hereditary torsion classes in $R$-$\Mod$, in particular we can identify $\mathbb{D}(R)$ with $R$-$\tors$. 

As an example, consider the quasi-dimension function $\upxi\in Q$-$\dim(R,\mathbb{D(R))}$, take an left $R$-module $M$ and the induced $\mathbb{D}(R)$-aspect $[M](\upxi)=\xi_{M}$.

The situations described by Proposition \ref{d15} and Corollary \ref{d16} are precisely the ones exposed in Golan lecture notes \cite{7}, but in our context. 
  In the inflator context, for example, consider any basic class of modules $\mathcal{B}\in\mathbb{B}(R)$, and let $\mathcal{C}rt(\mathcal{B})$ be the class of modules $M$ such that for all $N\subseteq M$, $N=0$ or $M/N\in\mathcal{B}$.  The $\mathcal{C}rt$-filtration in $\mathbb{D}(R)$ is just the well-known Gabriel filtration given by the pre-nucleus $\mathcal{G}ab:=\mathbb{D}vs\circ\mathcal{C}rt$ in $\mathbb{D}(R)$.

Given a class of modules, $\mathcal{B}$ and a module $M$, as in \cite{26} we define  the {\it slice} of $\mathcal{B}$ by $M$ as the set  $\langle M\rangle(\mathcal{B})$ by $[K,L]\in\langle M\rangle(\mathcal{B})\Leftrightarrow L/K\in\mathcal{B}$. We know that 
\begin{enumerate} 
\item If $ \mathcal{B}\in\mathbb{B}(R)$, then $\left\langle M\right\rangle(\mathcal{B})\in\EuScript{B}(\Lambda(M))$. 
\item If $\mathcal{C}\in\mathbb{C}(R)$, then $\left\langle M\right\rangle(\mathcal{C})\in\EuScript{C}(\Lambda(M))$. 
\item If $\mathcal{D}\in\mathbb{D}(R)$, then $ \left\langle M\right\rangle(\mathcal{D})\in\EuScript{D}(\Lambda(M))$. 
\end{enumerate}

Slicing by a module $M$ determines a morphism of frames 
$$\left\langle M\right\rangle(\_):\mathbb{B}(R)\longrightarrow\EuScript{B}(\Lambda(M)),$$ 
in fact, we have the morphism  
$$\left\langle M\right\rangle(\_):\mathbb{D}(R)\longrightarrow\EuScript{D}(\Lambda(M)).$$

Also from Proposition \ref{d10} we have,  for the idiom $\Lambda(M)$, a morphism of complete lattices, $\EuScript{N}_{a}:\rm{App}(\Lambda(M),\Gamma)\longrightarrow\EuScript{D}(\Lambda(M))$, where $\EuScript{N}_{a}:=\EuScript{D}vs\circ\EuScript{M}_{a}$ for each $a\in\Gamma$. 
Now from \cite{7}, for each $a\in\Gamma$ there is a morphism of complete lattices $\mathcal{N}_{a}:Q\rm{-dim}(R,\Gamma)\longrightarrow \mathbb{D}(R)^{op}$ given by $M\in\EuScript{T}_{\mathcal{N}(D,a)}$ if and only if $D(M)\leq a$. With all these we can state the following:

\begin{thm}
\label{d19}
Let $\Gamma$ be a complete lattice, and $R$-$\Mod$ the category of modules over a ring $R$. For each $a\in\Gamma$ and each $R$-module $M$ the following square 
$$\xymatrix{ Q\text{-}\dim(R,\Gamma)\ar[r]^{[M](\_)}\ar[d]_{\mathcal{N}_{a}} &  \rm{App}(M,\Gamma)\ar[d]^{\EuScript{N}_{a}} \\ \mathbb{D}(R)^{\text{\rm op}}\ar[r]_{\left\langle M\right\rangle(\_)} & \EuScript{D}(\Lambda(M))^{\text{\rm op}} \\  }$$ 
commutes in the category of complete lattices.
\end{thm}

\begin{proof}
We already know that all morphisms in the diagram are morphisms in the category of complete lattices, hence we just  have to check the commutativity. Consider any $D\in Q\text{-}\dim(R,\Gamma)$, and recall that $\EuScript{N}_{a}([M](D))=\EuScript{D}vs(\EuScript{M}_{a}(D_{M})$. Then, $[K,L]\in\EuScript{N}_{a}([M](D))=\EuScript{D}vs(\EuScript{M}_{a}(D_{M})$ if and only if   for all  $K\leq H<L$ there exists  $H<N\leq M$  such that  $N/H\in\EuScript{M}_{a}(D_{M})$. That is, $D_{M}([H,N])=D(N/H)\leq a$. But this is the condition for  $N/H$ belonging to $\EuScript{T}_{\mathcal{N}_{a}(D)}$, and this happens precisely when $[H,N]\in\left\langle M\right\rangle(\mathcal{N}_{a}(D))$.  Thus, $[K,L]\in\left\langle M\right\rangle(\mathcal{N}_{a}(D))$. Therefore, $\EuScript{N}_{a}([M](D))=\left\langle M\right\rangle(\mathcal{N}_{a}(D))$, for all $D$, and  the diagram commutes. 
\end{proof}

\begin{cor}
\label{d20}
Let $\Gamma$ be a complete lattice, and consider the complete lattice $\infty(\Gamma)$.  Let $R$-$\Mod$ be the category of modules over a ring $R$. Then, for any $a\in\infty(\Gamma)$ and any left $R$-module $M$, the following diagram commutes in the category of complete lattices:
$$\xymatrix{ Q\text{-}\dim(R,\infty(\Gamma))\ar[r]^{[M](\_)}\ar[d]_{\mathcal{N}_{a}} &  \rm{App}(M,\infty(\Gamma))\ar[d]^{\EuScript{N}_{a}} \\ \mathbb{D}(R)^{\text{\rm op}}\ar[r]_{\left\langle M\right\rangle(\_)} & \EuScript{D}(\Lambda(M))^{\text{\rm op}}.  }$$ 
\qed
\end{cor}

%As an example of this, 
\begin{ej}
Let $\Gamma$ be a complete lattice and consider $D\in Q\text{-}\dim(R,\Gamma)$, $a\in\Gamma$. Following   \cite[Chapter 12]{7}  define a sequence $k_{D,a}$, called the \emph{Gabriel filtration of } $(D,a)$, mimicking the one defined after our Corollary   \ref{d14}:

\begin{itemize}
\item[(1)] $k_{D,a}(0)=a$.

\item[(2)] If $0< \iota<\propto(\Gamma)$, then 

\begin{align*}
k_{D,a}(\iota)&=k_{D,a}(\iota-1)\vee(\bigvee\{D(M)\,|\, M\text{ is an } \mathcal{N}(D,k_{D,a}(\iota-1)))\text{-cocritical left }\\
&\qquad\qquad \qquad\qquad \qquad\qquad \qquad\qquad \qquad\qquad \qquad\qquad\qquad\text{module}\}).
\end{align*}

\item[(3)] If $0< \iota<\propto(\Gamma)$ is a limit ordinal, then $k_{D,a}(\iota)=\bigvee\left\{k_{D,a}(\lambda)\;|\; \lambda< \iota\right\}$.
\end{itemize}
Then, taking this filtration and applying the construction in Proposition \ref{d13},  in this context, we obtain the quasi-dimension function $\uplambda(D,k_{D,a})$ in  $Q\text{-}\dim(R,\infty(\Gamma))$. Now consider any left $R$-module $M$ and $[M](\uplambda(D,k_{D,a}))=\uplambda(D,k_{D,a})_{M}\in\rm{App}(M,\infty(\Gamma))$. Then, the $\infty(\Gamma)$-aspect for $M$ is just the $(D_{M},a)$-dimen\-sion. 
\end{ej}

\end{document}